\documentclass[reqno]{amsart}

\usepackage{amssymb,amsfonts, latexsym,amsmath,hhline,array,longtable,enumitem}
\usepackage{pdfsync,color, comment,colortbl, mathrsfs,stmaryrd,cite,graphicx,lscape}

\usepackage{ifpdf}
\ifpdf 
\usepackage[colorlinks=true, citecolor=blue, linkcolor=blue, urlcolor=blue]{hyperref} 
\fi

\newtheorem{formula}{}[section]
\newtheorem{definition}[formula]{Definition}
\newtheorem{corollary}[formula]{Corollary}
\newtheorem{remark}[formula]{Remark}
\newtheorem{lemma}[formula]{Lemma}
\newtheorem{theorem}[formula]{Theorem}

\newtheorem{problem}[formula]{Problem}

\newtheorem{proposition}[formula]{Proposition}

\newcommand{\R}{\mathcal{R}}
\newcommand{\F}{\mathbb{F}}
\newcommand{\Fq}{\F_q}
\newcommand{\rk}{\operatorname{rk}}

\newcommand{\mC}{\mathcal{C}}
\DeclareMathOperator{\sym}{Sym}

\newcommand{\Bil}{\operatorname{Bil}}

\newcommand{\bfn}{\mathbf {n}}
\newcommand{\bfm}{\mathbf {m}}
\newcommand{\spac}{\F^{\bfn\times\bfm}_q}
\newcommand{\srk}{\mathrm{srk}}

\usepackage{stmaryrd}
\newcommand{\db}[1]{\llbracket #1 \rrbracket}
\newcommand{\floor}[1]{{\left\lfloor{#1}\right\rfloor}}

\def\cX{{\mathcal X}}
\def\cY{{\mathcal Y}}

\newcommand{\sR}{\mathsf{R}}
\newcommand{\sT}{\mathsf{T}}

\DeclareMathOperator{\WL}{\mathsf{WL}}
 
\begin{document}

\title{A linear programming bound \\for sum-rank metric codes}
\author{Aida Abiad}
\address{Department of Mathematics and Computer Science, 
Eindhoven University of Technology, The Netherlands} 
\email{a.abiad.monge@tue.nl}
\author{Alexander L. Gavrilyuk}
\address{Interdisciplinary Faculty of Science and Engineering, 
Shimane University, Matsue, Japan}
\email{gavrilyuk@riko.shimane-u.ac.jp}
\author{Antonina P. Khramova}
\address{Department of Mathematics and Computer Science, 
Eindhoven University of Technology, The Netherlands} 
\email{a.khramova@tue.nl}
\author{Ilia Ponomarenko}
\address{Hainan University, Haikou, China}
\email{iliapon@gmail.com}

\renewcommand{\shortauthors}{Abiad, Gavrilyuk, Khramova, Ponomarenko}
\renewcommand{\shorttitle}{A linear programming bound for sum-rank metric codes}

\date{}

\begin{abstract}
We derive a linear programming bound on the maximum cardinality of error-correcting codes in the sum-rank metric. Based on computational experiments on relatively small instances, 
we observe that the obtained bounds outperform all previously known bounds.\\

\end{abstract}

\maketitle

\section{Introduction}

Sum-rank metric codes have received a considerable amount of attention over the last years, especially because of their performance of multi-shot network coding; see~\cite{booksrk,NU2010,BZ2024,byrne2021fundamental,byrne2022anticodes,AAA2024,chen2023new,MP18,MP19,MP20a,MPK19,MPK19a,MP20,NSZ2023,BP2021,ott2022covering}. A sum-rank metric code is a subset of matrix tuples of fixed dimensions, with every matrix being defined over the same finite 
field $\F_q$ (see Definition \ref{def:srk}). The distance between 
two such tuples is the sum of ranks of differences between 
respective elements of the tuples. 
As such, it is a generalization of both rank-metric codes and Hamming codes.

A central question in coding theory, concerning the maximum cardinality of a code with a given minimum distance, has recently been addressed in the context of the sum-rank metric. In~\cite{byrne2021fundamental}, properties of sum-rank metric codes are investigated, and several bounds based on classical coding theory approaches are shown. In~\cite{AAA2024}, the authors consider the sum-rank metric space from an algebraic graph theoretical point of view, introduce a sum-rank metric graph, and propose an eigenvalue bound, referred to as Ratio-type bound, by using a connection between the cardinality of a code with a minimum distance~$k+1$ and the $k$-independence number of the graph. This bound is calculated using the eigenvalues of the adjacency matrix of the sum-rank metric graph. An explicit closed formula is known for the case when the minimum distance is $3$ \cite{ACF2019} and $4$ \cite{KN2022}, and otherwise it can be calculated via a linear programming (LP) method which uses the so-called minor polynomials~\cite{F2020}. A geometrical approach to sum-rank metric codes using subspace designs is investigated in~\cite[Section 5]{SZ2022}.

While a classical approach to estimate the maximum cardinality of a code is to use Delsarte's LP method (which was first introduced in the context of the Hamming metric \cite{DelsarteLP}), this has not yet been developed for sum-rank metric codes. The idea behind this method is to consider 
a code as a subset of points of an 
association scheme, and then formulate an optimization 
problem where the objective is to maximize the size of 
a code subject to linear constraints derived 
by leveraging the properties of the association scheme.
By solving this linear program, one can obtain upper bounds 
on the number of codewords. 
This, combined with the duality of linear programming, provides one of the most powerful methods for bounding the size of codes arising from association schemes. Delsarte's LP approach has already been successfully applied to several metrics, \emph{e.g.}, for Hamming codes \cite{DelsarteLP}, rank-metric codes \cite{D1978}, bilinear alternating forms \cite{delsarte1975alternating}, 
permutation codes \cite{DIL2020}, and Lee codes \cite{A1982,S1986}, 
but also for newer metrics like subspace codes \cite{R2010}.


Despite its previous success, the LP method has not yet been applied to the sum-rank metric. This is perhaps due to the fact that, unlike for the Hamming or rank metric, the sum-rank metric does not give rise to an association scheme in a straightforward way by taking the distance relations between the codewords. 
In particular, except some special cases, the sum-rank metric graphs are 
not distance-regular~\cite[Proposition~11]{AAA2024}, unlike the Hamming graphs and 
the bilinear forms graphs 
corresponding to the rank-metric space.
 
In this paper, we develop an approach to utilizing 
the Delsarte's LP method to estimate the maximum 
cardinality of a sum-rank metric code with a given minimum distance. From computational experiments for small values of the parameters, we observe that the new LP bound outperforms all previously known bounds for sum-rank metric codes given in \cite{byrne2021fundamental} and \cite{AAA2024}. To prove our results, we describe a way to construct an association scheme for a sum-rank metric space. We do so by making use of the structure of the sum-rank metric graph as a Cartesian product of smaller rank-metric graphs~\cite[Proposition 9]{AAA2024}, and considering a direct product of association schemes of the Cartesian factors. Moreover, we investigate the coherent closure of a sum-rank metric graph and discuss conditions under which it is equivalent to the association scheme we construct. 

The structure of this paper is as follows. Section~\ref{sec:prelim} provides the preliminaries for the sum-rank metric space, known bounds on the size of sum-rank metric codes, as well as reminds some necessary background 
on association schemes and Delsarte's LP method. Section~\ref{sec:coco} introduces a way to construct a suitable association scheme for a sum-rank metric graph using the notion of a direct product of association schemes. Such association scheme is then used to derive the LP bound. 
In addition, particular attention is given to the coherent closure of the graph, and the conditions under which it coincides with the direct product of the coherent closures of the Cartesian factors of the graph. These conditions are also discussed in the more general context of graphs that can be seen as a Cartesian product. 
Finally, Section~\ref{sec:boundcomp} contains a computational comparison of the new LP bound against all the previously known bounds for sum-rank codes.

\section{Preliminaries}\label{sec:prelim}

In this section, we will recall some definitions 
and results on sum-rank metric codes and association schemes.

\subsection{Basic notation}
For a positive integer $n$, let $[n]$ denote the set $\{1,\dots,n\}$ and put $\db{n} := \{0\}\cup[n]$. 

Let $\Omega$ be a finite set and $R$ a binary relation on $\Omega$. The \textbf{adjacency matrix} of $R$ is a matrix $A\in\mathbb{R}^{\Omega\times\Omega}$ defined as follows: $(A)_{\alpha,\beta}=1$ if $(\alpha,\beta)\in R$ and $0$ otherwise.

By a \textbf{graph} we mean a finite simple undirected graph, \textit{i.e.}, a pair $\Gamma=(\Omega,E)$ of a finite set $\Omega$ of vertices and a set $E$ of unordered pairs from $\Omega$, $E\subseteq {\Omega\choose 2}$, called edges. Note that the edge set $E$ can also be identified with 
an irreflexive symmetric (\textbf{adjacency}) relation on~$\Omega$. Then the adjacency matrix of a graph $\Gamma$ is the adjacency matrix of~$E$.

Recall that the \textbf{Cartesian product} of graphs $\Gamma_1=(\Omega_1,E_1)$ and $\Gamma_2=(\Omega_2,E_2)$ is the graph\footnote{Note that in \cite{AAA2024}, the notation $\Gamma_1\times \Gamma_2$ was used to denote the Cartesian product. However, $\Box$ appears to be more common, \emph{e.g.}, see \cite{ProductBook}.} 
$\Gamma_1\Box \Gamma_2$ whose vertex set is 
$\Omega_1\times \Omega_2$, where vertices 
$(\alpha_1,\alpha_2)$ and $(\beta_1,\beta_2)$ are adjacent 
if either $\alpha_1=\beta_1$ and $(\alpha_2,\beta_2)\in E_2$, 
or $\alpha_2=\beta_2$ and $(\alpha_1,\beta_1)\in E_1$. 
If $A_i$, $i=1,2$, denotes the adjacency matrix 
of a graph $\Gamma_i$, and $A$ is the adjacency matrix 
of $\Gamma_1\Box \Gamma_2$, then  
\begin{equation}\label{eq:Acartprod}
    A=A_1\otimes I_{\Omega_2}+I_{\Omega_1}\otimes A_2,
\end{equation}
where $\otimes$ means the Kronecker product of matrices, 
and $I_{\Omega}$ is the identity matrix of $\mathbb{R}^{\Omega\times\Omega}$. 
This definition can be inductively extended 
to a Cartesian product of any finite number of graphs.

\subsection{Sum-rank metric and sum-rank metric graphs}

Let $t$ be a positive integer and let ${\bf n}=(n_1,\ldots,n_t)$, ${\bf m}=(m_1,\ldots,m_t)$ be tuples of positive integers 
with $m_1 \geq m_2 \geq \cdots \geq m_t$, 
and $m_i\geq n_i$ for all $i\in [t]$. 

For a prime power $q$ and 
positive integers $m\geq n$, 
let $\mathbb{F}_q^{n\times m}$
denote the vector space of all $n×\times m$ matrices over the finite field $\mathbb{F}_q$. 
Denote by $\rk(M)$ the rank of a matrix 
$M\in \mathbb{F}_q^{n\times m}$.

 \begin{definition}\label{def:srk}
The \emph{\textbf{sum-rank metric space}} is 
an $\mathbb{F}_q$-linear vector space 
$\mathbb{F}_q^{{\bf n}\times {\bf m}}$ defined as follows:
\[
\mathbb{F}_q^{{\bf n}\times {\bf m}}:=
\mathbb{F}_q^{n_1\times m_1}\times \cdots \times
\mathbb{F}_q^{n_t\times m_t},
\]
where $\times$ stands for the direct product of vector spaces.
The \emph{\textbf{sum-rank}} of an element $X=(X_1,\ldots, X_t)\in \mathbb{F}_q^{{\bf n}\times {\bf m}}$ 
is $\srk(X) = \sum_{i=1}^t \rk(X_i)$.
The \emph{\textbf{sum-rank distance (metric)}} between 
$X,Y \in \mathbb{F}_q^{{\bf n}\times {\bf m}}$ is $\srk(X - Y)$.
In case $t=1$ we sometimes refer to the sum-rank distance as simply the \emph{\textbf{rank distance}}.
\end{definition}

It is easy to see that the sum-rank distance is indeed a distance on $\mathbb{F}_q^{{\bf n}\times {\bf m}}$.

\begin{definition}
A \emph{\textbf{sum-rank metric code}} is a non-empty subset $\mC \subseteq \mathbb{F}_q^{{\bf n}\times {\bf m}}$. 
The \emph{\textbf{minimum sum-rank distance}} of a code $\mC$ with $|\mC| \geq 2$ is defined 
by $$\srk(\mC) := \min\left\{\srk(X-Y)\colon X, Y \in \mC, \, X \neq Y\right\}.$$
\end{definition}

The following definition, which is introduced in \cite{AAA2024}, allows one to study sum-rank metric codes from the graph-theoretical perspective. 

\begin{definition}\label{def:srkgraph}
    The \emph{\textbf{sum-rank metric graph}} $\Gamma\left(\mathbb{F}_q^{{\bf n}\times {\bf m}}\right)$ is a graph whose vertex set is $\mathbb{F}_q^{{\bf n}\times {\bf m}}$, and two vertices $X,Y$ are adjacent if and only if $\srk(X-Y)=1$.
\end{definition}

 We sometimes omit $\mathbb{F}_q^{{\bf n}\times {\bf m}}$ and simply write $\Gamma$ when it is clear which sum-rank-metric space the graph corresponds to.

A simple yet important observation is that, for any two vertices $X,Y$ of the graph~\mbox{$\Gamma\left(\mathbb{F}_q^{{\bf n}\times {\bf m}}\right)$,} 
the geodesic distance between them coincides with the sum-rank distance $\srk(X-Y)$~\cite[Proposition 4.3]{byrne2022anticodes}.

The following two special cases covered by Definition~\ref{def:srk} have been long known and well-studied in the coding theory 
literature: in case $n_i=m_i=1$ for all $i\in[t]$, it is equivalent to 
the Hamming metric on the Hamming space $\Fq^t$, 
while in case $t=1$, the sum-rank metric reduces to the rank metric on 
the space $\mathbb{F}_q^{n\times m}$. 
Hence the sum-rank metric is a generalization of both the Hamming metric 
and the rank metric. In the former case, 
the corresponding sum-rank metric graphs 
are well known in the literature 
as Hamming graphs, denoted $H(t,q)$, while in the latter 
case those are bilinear forms graphs.
Recall that the \textbf{bilinear forms graph}, sometimes 
denoted by $\Bil_q(n,m)$, 
is a graph with $\mathbb{F}_q^{{n}\times {m}}$ as the vertex set, where two vertices (matrices) 
are adjacent if and only if the 
rank distance between them is exactly one.

A graph $\Gamma$ with diameter $D$ is said to be \textbf{distance-regular} if for all $i,j,k\in\db{D}$ and for each pair of vertices $x$ and $y$ at distance $k$ from each other the number of vertices that are at distance $i$ from $x$ and at distance $j$ from $y$ is equal to a constant $p^k_{i,j}$ that does not depend on the choice of vertices $x$ and $y$.

Note that both Hamming graphs and bilinear forms graphs are distance-regular, see~\cite{BCNbookDRG}.

\begin{proposition} \cite[Proposition~9]{AAA2024}\label{prop:prod}
    Let $\Gamma_i=\Gamma\left(\mathbb{F}_q^{{n}_i\times {m}_i}\right)$ for $i\in [t]$. 
    Then $\Gamma_i$ is isomorphic to the bilinear forms graph $\Bil_q(n_i,m_i)$ and 
    the sum-rank metric graph $\Gamma\left(\mathbb{F}_q^{{\bf n}\times {\bf m}}\right)$ is the Cartesian product $\Gamma_1\Box\cdots\Box\Gamma_t$. 
\end{proposition}

Graphs that are Cartesian products of $t$ distance-regular graphs are examples of $t$-distance-regular graphs considered, \emph{e.g.}, in~\cite{BVZZ2023}.

Note that in case $n_i=m_i=1$ for all $i\in[t]$, the graph $\Bil_q(1,1)$ is 
a complete graph on $q$ vertices, 
and $\Gamma\left(\mathbb{F}_q^{{\bf n}\times {\bf m}}\right)$ is the Cartesian 
product of $t$ complete graphs on $q$ 
vertices, which is the Hamming graph 
$H(t,q)$.

\subsection{Coding theory bounds}

Several 
bounds on the maximum size of a sum-rank metric code $\mC\subseteq\F^{\bfn\times\bfm}_q$ with $|\mC|\geq2$ and 
a given minimum distance $\srk(\mC)\geq d$ were introduced in~\cite{byrne2021fundamental}. The first four bounds are induced by the following connection to the Hamming-metric 
case: 
$|\mC|$ is upper-bounded by the size of a Hamming code over $\F_{q^m}$ of length $N$ and minimum distance at least $d$, where $m=\max_{i\in[t]} m_i$ and $N = \sum_{i\in[t]} n_i$.

\begin{theorem}[see~\text{\cite[Theorem III.1]{byrne2021fundamental}}]\label{thm:induced} Let $m=\max_{i\in[t]} m_i$ and let $\mC\subseteq\spac$ be a sum-rank metric code with $|\mC|\geq 2$ and $\srk(\mC) \ge d$. The following hold:
\begin{center}
    \begin{tabular}{ll}
     \textbf{Induced Singleton bound:} & $|\mC|\leq q^{m(N-d+1)}$, \\[0.5cm]
     \textbf{Induced Hamming bound:} & $|\mC|\leq \floor{\frac{q^{mN}}{\sum_{s=0}^{\floor{(d-1)/2}}\binom{N}{s}(q^m-1)^s}}$, \\[0.5cm]
     \textbf{Induced Plotkin bound:} & $|\mC|\leq \floor{\frac{q^{m}d}{q^md-(q^m-1)N}}$ if $d>(q^m-1)N/q^m$, \\[0.5cm]
     \textbf{Induced Elias bound:} & $|\mC|\leq \floor{\frac{Nd(q^m-1)}{q^mw^2-2Nw(q^m-1)+(q^m-1)Nd}\cdot\frac{q^{mN}}{V_w(\F^N_{q^m})}}$. \\[0.5cm]
\end{tabular}
\end{center}
In the Induced Elias bound, $w$ is any integer between $0$ and $N(q^m-1)/q^m$ such that the denominator is positive, and $V_w(\F_{q^m}^N)=\sum_{i=0}^w {N\choose i}(q^m-1)^i$, \textit{i.e.}, the cardinality of any ball of radius $w$ in $\F_{q^m}^N$ with respect to the Hamming distance.
\end{theorem}

The other four bounds presented in~\cite{byrne2021fundamental} are not induced by the Hamming-metric case and are specific to the sum-rank metric.

\begin{theorem}[see \text{\cite[Theorems III.2, III.6--III.8]{byrne2021fundamental}}]\label{thm:non-induced} Let $\mC\subseteq\spac$ be a sum-rank metric code with $|\mC|\geq2$ and $\srk(\mC)\ge d$. Let $j$ and $\delta$ be the unique integers satisfying $d-1=\sum_{i=1}^{j-1} n_i + \delta$ and $0\leq\delta\leq n_j-1$. Let $\ell\leq t-1$ and $\delta'\leq n_{\ell+1}-1$ be the unique positive integers such that $\smash{d-3=\sum_{j=1}^\ell n_j+\delta'}$.
Define $\smash{\bfn'=(n_{\ell+1}-\delta',n_{\ell+2},\dots,n_t)}$ and $\bfm'=(m_{\ell+1},m_{\ell+2},\dots,m_t)$. Finally, let $Q=\sum_{i=1}^t q^{-m_i}$.
The following hold:
\begin{center}
    \begin{tabular}{ll}
     \textbf{Singleton bound:} & $|\mC|\leq q^{\sum_{i=j}^t m_in_i-m_j\delta}$, \\[0.5cm]
     \textbf{Sphere-Packing bound:} & $|\mC|\leq\floor{\frac{|\spac|}{V_r(\spac)}}$, where $r=\floor{(d-1)/2}$, \\[0.5cm]
     \textbf{Projective Sphere-Packing bound:} & $|\mC|\leq\floor{\frac{|\F^{\bfn'\times\bfm'}_q|}{V_1(\F^{\bfn'\times\bfm'}_q)}}$ if $3\leq d\leq N$, \\[0.5cm]
     \textbf{Total Distance bound:} & $|\mC|\leq\frac{d-N+t}{d-N+Q}$ if $d>N-Q$. \\[0.5cm]
\end{tabular}
\end{center}
In the Sphere-Packing and Projective Sphere-Packing bounds, the denominator
denotes the cardinality of any ball in the sum-rank metric of radius $r$. For example,
$$V_r(\spac)=|\{(X_1,\dots,X_t)\in\spac \mid \srk(X_1,\dots,X_t)\leq r\}|,$$
and a closed formula is provided in~\cite{byrne2021fundamental}.
\end{theorem}

\subsection{Eigenvalue bounds}
In \cite[Corollary 16]{AAA2024} it was shown that any upper bound on the $k$-independence number of a graph yields an upper bound on the size of a sum-rank metric code with minimum distance~\mbox{$k+1$}, and vice versa. The next result provides the main bound we will be using to compare our new LP bound with; the so-called Ratio-type LP bound, which was recently proposed to be used for sum-rank metric codes in \cite{AAA2024}.

From here on, let $\mathbb{R}_k[x]$ denote the set of polynomials in variable $x$ with coefficients in $\mathbb{R}$ of degree at most $k$.

\begin{theorem}[Ratio-type bound; see \cite{ACF2019}]\label{thm:hoffman-like} Let $\Gamma=(\Omega,E)$ be a regular graph with $n$ vertices and adjacency matrix $A$ with eigenvalues \mbox{$\lambda_1\geq\dots\geq\lambda_n$.} Let \mbox{$p\in\mathbb{R}_k[x]$.} Define \mbox{$W(p)=\max_{u\in \Omega}\{(p(A))_{uu}\}$} and \mbox{$\lambda(p)=\min_{i=2,\dots,n}\{p(\lambda_i)\}$.} Then
\begin{equation}\label{eq:hoffman-like}
    \alpha_k(\Gamma) \leq n \, \frac{W(p)-\lambda(p)}{p(\lambda_1)-\lambda(p)}.
\end{equation}
\end{theorem}

In cases $k=2$ and $k=3$, the best polynomial $p\in\R_k[x]$ that minimizes the bound in Theorem~\ref{thm:hoffman-like} was shown in~\cite{ACF2019} and \cite{KN2022}, respectively.
\begin{theorem}[Ratio-Type bound, $k=2$; see~\cite{ACF2019}]\label{thm:hoffman-k=2} 
Let $\Gamma$ be a regular graph with $n$ vertices and $r\geq 2$
\emph{distinct} eigenvalues \mbox{$\theta_0>\theta_1>\dots>\theta_r$}
of the adjacency matrix. Let $\theta_i$ be the largest eigenvalue such that $\theta_i\leq -1$. Then $$\alpha_2(\Gamma)\leq n\frac{\theta_0+\theta_i\theta_{i-1}}{(\theta_0-\theta_i)(\theta_0-\theta_{i-1})}.$$ Moreover, this is the best possible bound that can be obtained by choosing a polynomial via Theorem~\ref{thm:hoffman-like}.
\end{theorem}

\begin{theorem}[Ratio-Type bound, $k=3$; see~\cite{KN2022}] \label{thm:hoffman-k=3} Let $\Gamma$ be a regular graph with 
the vertex set $\Omega$, $|\Omega|=n$, and 
with $r\geq 3$ \emph{distinct} eigenvalues \mbox{$\theta_0>\theta_1>\dots>\theta_r$}
of the adjacency matrix. 
Let $s$ be the largest index such that \mbox{$\smash{\theta_s\geq -\frac{\theta_0^2+\theta_0\theta_r-\Delta}{\theta_0(\theta_r+1)}}$,} where \mbox{$\Delta=\max_{u\in \Omega}\{(A^3)_{uu}\}$.} Then
$$\alpha_3(\Gamma)\leq n\frac{\Delta-\theta_0(\theta_s+\theta_{s+1}+\theta_r)-\theta_s\theta_{s+1}\theta_r}{(\theta_0-\theta_s)(\theta_0-\theta_{s+1})(\theta_0-\theta_r)}.$$ Moreover, this is the best possible bound that can be obtained by choosing a polynomial via Theorem~\ref{thm:hoffman-like}.
\end{theorem}

The following lemma is a new contribution to the mentioned eigenvalue bound; it provides conditions under which the Ratio-type bound of Theorem~\ref{thm:hoffman-like} is tight.

\begin{lemma}\label{l:rtbound-eq}
Let $\Gamma$, $W(p)$, $\lambda(p)$ be as in Theorem~\ref{thm:hoffman-like}. If the Ratio-type bound of Theorem~\ref{thm:hoffman-like} is met by a $k$-independent set $U$ of the graph $\Gamma$, then the following conditions hold:
\begin{enumerate}[label=(\roman*)]
    \item $W(p) = \frac1{|U|} \sum\limits_{u\in U} (p(A))_{uu}$,\label{itm:Wp}
    \item Any vertex $u\in U$ has $W(p)$ neighbors in $U$; in particular, the induced subgraph on vertices of $U$ is regular,\label{itm:uinU}
    \item Any vertex $v\notin U$ has $W(p)-\lambda(p)$ neighbors in $U$.\label{itm:unotinU}
\end{enumerate}
\end{lemma}

\begin{proof}
The proof follows directly from analyzing the inequality that arises as a part of the proof of~\cite[Theorem 3.2]{ACF2019}, which we now give.
Let $A$ denote the adjacency matrix of $\Gamma$ with adjacency eigenvalues $\lambda_1\geq \cdots \geq \lambda_n$. Let $\mu_1\geq \mu_2$ denote the eigenvalues of the quotient matrix of the partition
of $\Omega$ into $U$ and $\Omega\setminus U$:
$$\left(
\begin{matrix}
    \frac1r \sum\limits_{u\in U} (p(A))_{uu} & p(\lambda_1) - \frac1r \sum\limits_{u\in U} (p(A))_{uu} \\
    \frac{rp(\lambda_1) - \sum\limits_{u\in U} (p(A))_{uu}}{n-r} & p(\lambda_1) - \frac{rp(\lambda_1) - \sum\limits_{u\in U} (p(A))_{uu}}{n-r} \\
\end{matrix}
\right).$$

Finally, let $\mu_1 = p(\lambda_1)$, and $r = |U| = \alpha_k(\Gamma)$. Then,

\begin{align}
\label{eq:rtbound-proof1} \lambda(p) &\leq \mu_2 \\
\nonumber &= \frac1r \sum\limits_{u\in U} (p(A))_{uu} - \frac{rp(\lambda_1)-\sum\limits_{u\in U} (p(A))_{uu}}{n-r} \\
\label{eq:rtbound-proof3} &\leq W(p) - \frac{rp(\lambda_1)-rW(p)}{n-r}.
\end{align}
The explicit bound on $r$ is then deduced from the inequality between $\lambda(p)$ and the final expression.

In case the Ratio-type bound is tight, the inequalities\:\eqref{eq:rtbound-proof1} and\:\eqref{eq:rtbound-proof3} must be equalities. In particular, it means that $$W(p)=\frac1r \displaystyle\sum\limits_{u\in U} (p(A))_{uu},$$ which gives~\ref{itm:Wp}.
The inequality\:\eqref{eq:rtbound-proof1} being an equality means that the interlacing of eigenvalues is tight, and by \cite[Corollary 2.3]{Haemers95} it follows that the partition of $\Gamma$ into $U$ and $\Omega\setminus U$ is equitable (regular). It follows that each vertex from $\Omega\setminus U$ is adjacent to $$\frac{rp(\lambda_1)-\sum\limits_{u\in U} (p(A))_{uu}}{n-r} = \frac{rp(\lambda_1)-r W(p)}{n-r} = W(p) - \lambda(p)$$ vertices from $U$, which gives~\ref{itm:unotinU}, and each vertex from $U$ is adjacent to $$\frac1r \sum\limits_{u\in U} (p(A))_{uu} = W(p)$$ vertices from $U$, which gives~\ref{itm:uinU}.
\end{proof}

A graph $\Gamma$ is \textbf{$l$-partially walk-regular} if for any vertex $v$ and any positive integer $i\leq l$ the number of closed walks of length $i$ does not depend on the choice of $v$. We also say that $\Gamma$ is \textbf{partially walk-regular} if it is $l$-partially walk-regular for some~$l$, and that $\Gamma$ is \textbf{walk-regular} if it is $l$-partially walk-regular for any~$l$.

\begin{remark}
The condition~\ref{itm:Wp} of Lemma~\ref{l:rtbound-eq} always holds in a partially walk-regular graph, since all entries of the diagonal of $p(A)$ have the same value. In particular, this condition holds if $\Gamma$ is a sum-rank metric graph, see \cite[Proposition~11]{AAA2024}.
\end{remark}

For a fixed $k$, the challenge behind applying the Ratio-type bound is to find a polynomial of degree $k$ which minimizes the right-hand of Eq.\:\eqref{eq:hoffman-like}. This was resolved for $k\in\{2,3\}$ resulting in closed formul{\ae} presented in Theorems~\ref{thm:hoffman-k=2} and~\ref{thm:hoffman-k=3}. However, even in the case $k=3$ finding the polynomial in general turns out to be an involved problem, with an entire paper devoted to it~\cite{KN2022}.
In \cite{F2020}, an LP implementation using the so-called minor polynomials was proposed for finding, for a given $k$ and a $k$-partially walk-regular graph $\Gamma$, the polynomial $p$ that optimizes the Ratio-type bound from Theorem~\ref{thm:hoffman-like}. We will use such LP 
to compute the Ratio-type bound and compare it with the new Delsarte's LP bound. In the LP for the Ratio-type bound, the inputs are the distinct adjacency eigenvalues of the graph $\Gamma$, denoted \mbox{$\theta_0>\cdots>\theta_r$,} with respective multiplicities $m(\theta_i)$, \mbox{$i\in\{0,\dots,r\}$.} The \textbf{minor polynomial} \mbox{$f_k\in\mathbb{R}_k[x]$} is a
polynomial that minimizes \mbox{$\sum_{i=0}^r m(\theta_i) f_k(\theta_i)$.} Let $p=f_k$ be defined by \mbox{$f_k(\theta_0)=x_0=1$} and \mbox{$f_k(\theta_i)=x_i$} for \mbox{$i\in\{1,\dots,r\}$,} where the vector \mbox{$(x_1,\dots,x_r)$} is a solution of the following linear program: 
\begin{equation}\label{eq:FiolLP}
\boxed{
\begin{array}{ll@{}l}
\text{minimize}  &\sum_{i\in \db{r}} m(\theta_i)x_i &\\
\text{subject to} &f[\theta_0,\dots,\theta_s]=0, &\quad s=k+1,\dots,r,\\
&x_i\geq0, &\quad i\in[r].\\
\end{array}
}
\end{equation}
Here, \mbox{$f[\theta_0,\dots,\theta_m]$} denote $m$-th divided differences of Newton interpolation, recursively defined by $$f[\theta_i,\dots,\theta_j]=\frac{f[\theta_{i+1},\dots,\theta_j]-f[\theta_i,\dots,\theta_{j-1}]}{\theta_j-\theta_i},$$ where $j>i$, starting with \mbox{$f[\theta_i]=x_i$} for \mbox{$i\in\{0,\dots,r\}$.}
In~\cite{F2020}, it was shown that, for $k=3$, 
using the LP\:\eqref{eq:FiolLP}
it is possible to obtain tight bounds for every Hamming graph $H(r,2)$, and it was also shown that for these graphs 
the Ratio-type bound coincides with Delsarte's LP bound~\cite{DelsarteLP}; see also Section \ref{sec:final}.


\subsection{Association schemes}\label{ssec:as}

The basic theory of association schemes and their 
Bose-Mesner algebras, given in this subsection for 
the sake of completeness, are standard and can be found 
in more detail, \emph{e.g.}, in~\cite[Chapter 2]{BCNbookDRG}, 
\cite{DelsarteLP}, or \cite{Godsil}. 


\begin{definition}\label{def:as}
    A \emph{\textbf{(symmetric) $D$-class association scheme}} $\cX=(\Omega,\sR)$ is a finite set $\Omega$ (of points) 
    together with a collection 
    $\sR=\{R_i\mid i\in \db{D}\}$ 
    of non-empty binary relations on $\Omega$, satisfying the following
    four conditions:
    \begin{enumerate}[label=(\roman*)]
        \item $\sR$ is a partition of $\Omega\times \Omega$. 
        \item 
        $R_0$ is the diagonal 
        of $\Omega\times \Omega$.
        \item Each binary relation $R_i$ equals its transpose (converse). 
        \item\label{itm:interray_as} For all $i,j,k\in \db{D}$ and $(\alpha,\beta)\in R_k$, the number of $\gamma\in \Omega$ 
        such that $(\alpha,\gamma)\in R_i$ and $(\gamma,\beta)\in R_j$ is a constant denoted by $p_{i,j}^k$ that does not depend on the choice of $(\alpha,\beta)$.
    \end{enumerate}
\end{definition}

Throughout this paper, we refer to a symmetric $D$-class association scheme simply as association scheme (informally, we will occasionally call 
an association scheme just a scheme). By the relations of~$\cX$ 
we mean the elements of $\sR$.
The number of relations $D+1$ and the constants $p_{i,j}^k$ 
are called the \textbf{rank} and the \textbf{intersection numbers} of a scheme, respectively. 

The following are two well-known examples 
of association schemes, which provide an 
algebraic framework to study codes in 
the Hamming and rank metrics.

\begin{definition}\label{def:hamming} Let $n$ and $q$ be positive integers. The \emph{\textbf{Hamming scheme}} is an association scheme $([n]^q,\sR)$ 
of rank $n+1$, where all pairs $x,y\in[n]^q$ with Hamming distance 
equal to $i$ are in the relation $R_i\in \sR$, $i\in \db{n}$.
\end{definition}

\begin{definition}\label{def:bf} Let $n,m$ be positive integers with $n\leq m$ 
and $q$ be a prime power. 
The \emph{\textbf{bilinear forms scheme}} 
is an association scheme $(\F_q^{n\times m},\sR)$ 
of rank $n+1$, where all pairs $A,B\in \F_q^{n\times m}$ 
with $\rk(A-B)=i$ are in the relation $R_i\in \sR$, 
$i\in \db{n}$.
\end{definition}

Let $\cX=(\Omega,\sR)$ be an association scheme, 
and $A_i\in \mathbb{R}^{\Omega\times \Omega}$ denote 
the adjacency matrix of $R_i\in \sR$.
Definition \ref{def:as}, which is stated in terms of 
binary relations on $\Omega$, 
can also be rewritten in terms of the adjacency matrices 
that they represent:
    \begin{enumerate}[label=(\roman*)]
        \item $\sum_{i\in I}A_i=J_{\Omega}$ 
        (the all-ones matrix).
        \item $A_0=I_{\Omega}$ (the identity matrix).  
        \item $A_i=A_i^{\top}$ for all $i\in \db{D}$. 
        \item\label{itm:prod_adjprops} $A_i A_j = \sum\limits_{k\in \db{D}} p_{i,j}^k A_k$ for all $i,j\in \db{D}$.
    \end{enumerate}
By~\ref{itm:prod_adjprops} above, the matrices $A_i$ form 
a basis for a $(D+1)$-dimensional matrix algebra 
over $\mathbb{C}$, called 
the \textbf{Bose-Mesner algebra} of $\cX$.
This algebra is closed under matrix multiplication, 
Schur (entrywise) multiplication, and transposition, 
which thus makes it an example of 
a \textbf{coherent algebra}.
Since the matrices $A_i$, $i\in \db{D}$, pairwise 
commute, it follows that $\mathbb{R}^{\Omega}$ 
decomposes into the orthogonal direct sum 
of their common maximal eigenspaces.

The Bose-Mesner algebra of $\cX$ has 
two distinctive linear bases, 
namely, the one consisting of the adjacency matrices $A_i$ 
and another one of the so-called 
\textbf{primitive idempotents} $E_j$, which 
are the orthogonal projection matrices 
onto maximal common eigenspaces of the adjacency matrices.
In particular, $E_i E_j$ is the zero matrix 
if $i\neq j$ and $E_i^2=E_i$ 
for all $i,j\in \db{D}$.
For expressing each basis in terms of the other, 
define the constants $P_{ji}$
and $Q_{ij}$ as  
\[
A_i = \sum\limits_{j\in \db{D}} P_{ji} E_j
\quad \text{and}\quad 
E_i = \frac1{|\Omega|}\sum\limits_{j\in \db{D}} Q_{ij} A_j.
\]
Now, the corresponding change-of-basis matrices 
$P=(P_{ji})$ and $Q=(Q_{ij})$ of order $D+1$ 
are respectively called the \textbf{first and second eigenmatrices} (or character tables)
of the Bose-Mesner algebra (or of the scheme). 
In particular, the $i$-th column of $P$ consists of the eigenvalues of $A_i$, \textit{i.e.}, 
$P_{ji}$ is the eigenvalue of $A_i$ 
on the $j$-th maximal common eigenspace. 
Furthermore, note that $Q=|\Omega|\cdot P^{-1}$.

Next, we briefly 
(see \cite[Chapter~9]{BCNbookDRG}
for further details)
recall how to compute the $P$ and $Q$ matrices for the bilinear forms scheme from Definition \ref{def:bf} with $D=n$. 
First, observe that the adjacency matrix $A_1$ is 
the adjacency matrix of the bilinear forms graph $\Bil_q(n,m)$ 
and has $D+1$ distinct eigenvalues 
given by 
\[
\theta_j = \frac1{q-1} \left( (q^{n-j}-1)(q^m-q^j)- q^j + 1 \right),\quad j\in \db{D}.
\]
Second, as the bilinear forms scheme is 
$P$-polynomial (see \cite{BCNbookDRG,DelsarteLP}), 
for every $i\in \db{D}$, 
the adjacency matrix $A_i$ equals $p_i(A_1)$, 
where $p_i$ is a polynomial 
of degree $i$ defined by 
the following recursive equation:
\[
p_i = \frac1{c_i} \left( (x - a_{i-1}) p_{i-1} - b_{i-2} p_{i-2}\right),\quad i\in \db{D}, 
\]
where $p_0=1$, $p_1(x)=x$, and 
\begin{align*}
b_i &:= p_{1,i+1}^i = \frac{q^{2i}(q^{m-i}-1)(q^{n-i}-1)}{q-1}, \\
c_i &:= p_{1,i-1}^i = \frac{q^{i-1}(q^i-1)}{q-1}, 
\\
a_i &:= p_{1,i}^i = b_0 - b_i - c_i.
\end{align*}
Furthermore, the bilinear forms scheme 
is self-dual (see \cite[Section 6.1]{BCNbookDRG}), 
which, in particular, means that the matrices $P$ and $Q$ are equal. 
Therefore, 
\begin{equation}\label{eq:Qbil}
P_{ji}=Q_{ji}=p_i(\theta_j),\quad i,j\in \db{D}.  
\end{equation}
For more information on bilinear forms scheme, the reader is referred to \cite{DelsarteLP,D1978}.

\subsection{Delsarte's linear programming bound}
In this subsection we briefly recall Delsarte’s linear programming method~\cite{DelsarteLP}.


Let $\cX=(\Omega,\sR)$ be 
an association scheme of rank $D+1$. 
Let $\Delta$ be a non-empty subset of $\Omega$. 
The (inner) \textbf{distribution vector} of $\Delta$ is a vector $\mathbf{a}=(a_0,\ldots,a_D)$ with entries 
\[
a_i = \frac{|(\Delta\times\Delta)\cap R_i|}{|\Delta|}
\]
for every $R_i\in \sR$. 
It is clear that $a_i\geq 0$ 
for all $i\in\db{D}$, 
$\sum\limits_{i\in \db{D}} a_i = |\Delta|$, 
and the distribution vector is normalized 
so that $a_0 = 1$.

A key observation for Delsarte's LP bound is the following result.

\begin{theorem}\label{theo:LP}
\cite[Theorem 3.3]{DelsarteLP}
With the above notation, the distribution 
vector $\mathbf{a}$ satisfies 
\[
\mathbf{a}Q\geq \mathbf{0},
\]
where $Q$ is the second eigenmatrix of $\cX$.
\end{theorem}

Then a linear programming bound for a code  
in $\cX$ is carried by Theorem \ref{theo:LP} as follows. 
For a subset of indices $M\subseteq \db{D}$, 
an \textbf{$M$-code} is a subset 
$\Delta\subseteq \Omega$ such that 
any two distinct points of $\Delta$ are in 
$R_i$ for some $i\in M$. In other words, 
$\Delta$ is an $M$-code if $a_i=0$ for 
all $i\notin M$. 
Consider the following LP formulation:
\begin{equation}\label{eq:LP}
\boxed{
\begin{array}{ll@{}ll}
\text{maximize}  &\sum_{i\in \db{D}} a_i &\\
\text{subject to} &\mathbf{a}Q\geq \mathbf{0},\\
&\mathbf{a}\geq \mathbf{0},\\
&a_0=1,\\
&a_i = 0, &\quad i\notin M.
\end{array}
}
\end{equation}

It follows from Theorem \ref{theo:LP} 
that the cardinality $|\Delta|$ of such an $M$-code 
is upper bounded by the solution of the LP\:\eqref{eq:LP}.

\section{An association scheme for sum-rank metric graphs}\label{sec:coco}
In this section, we define and study an association scheme related to sum-rank metric graphs. 
Recall that, in the case of the Hamming or bilinear forms association schemes 
(see Definitions \ref{def:hamming}, \ref{def:bf}), 
the adjacency matrix $A_1$ of the corresponding graph 
generates the respective Bose-Mesner algebra. 
In general, unlike these two partial cases, 
the sum-rank distance itself 
does not directly determine an association scheme.
However, one can define an association scheme based on the fact that the
 sum-rank metric graph is constructed as the Cartesian product of bilinear forms graphs \cite[Proposition~9]{AAA2024}. 
This scheme may possibly be larger (in the sense of Definition 
\ref{def:leq} below, see the discussion in Section \ref{sec:ASsrkgraph}) 
than we need, but it suffices to estimate Delsarte's LP bound 
on the size of sum-rank metric codes.

\subsection{Direct product of association schemes}
The concept of a direct product of association schemes is well known; 
here we follow the definition proposed in \cite{FT1985}.

\begin{definition}\label{def:direct}
Let 
$\cX_i = (\Omega_i,\sR_i)$ be an association scheme of rank $D_i$ with relations 
$\sR_i=\{R^i_j\mid j\in\db{D_i}\}$, $i\in[t]$. 
Then the \emph{\textbf{direct product}} 
$\cX_1\otimes\cdots\otimes\cX_t$ 
is an association scheme 
$\cX=(\Omega,\sR)$ where:
\begin{enumerate}[label=(\roman*)]
    \item $\Omega=\Omega_1\times \cdots \times \Omega_t$,
    \item $\sR=\{R_{(j_1\dots,j_t)}\mid 
    (j_1\dots,j_t)\in \db{D_1} \times\cdots\times \db{D_t}\}$, 
    where 
\[
R_{(j_1\dots,j_t)}=
\{
\left((\alpha_1,\ldots,\alpha_t),(\beta_1,\ldots,\beta_t)\right)\in 
\Omega^2\colon (\alpha_i,\beta_i)\in R_{j_i}^i\in \sR_i, i\in [t]\}.
\]
\end{enumerate} 
\end{definition}

The fact that the relations in $\sR$ indeed give an association 
scheme can be easily verified (see \cite[Section~2.2]{M1999}). 
In the more general context of coherent configurations 
\cite[Section~3.2.2]{CP2019}, which are not the focus of this paper (see Section~\ref{sec:ASsrkgraph}), the direct product is called 
the \textbf{tensor product}.

Note that the adjacency matrix $A_{(j_1\dots,j_t)}$ of the relation $R_{(j_1\dots,j_t)}\in\sR$ is the Kronecker 
product $A^1_{j_1}\otimes\cdots\otimes A^t_{j_t}$ 
of the adjacency matrices $A_{j_i}^i$  
with $j_i\in\db{D_i}$, $i\in [t]$.

\begin{lemma}\cite{M1999}
\label{l:Q}
In the notation of Definition \ref{def:direct}, 
the second eigenmatrix~$Q$ of~$\cX$ 
is the Kronecker product 
$Q_1\otimes Q_2\otimes\cdots\otimes Q_t$, where $Q_i$ is the second eigenmatrix of~$\cX_i$, $i\in [t]$.     
\end{lemma}

Using Definition \ref{def:direct}, we can consider 
an association scheme that contains the edge set of 
a sum-rank metric graph as a union of some of its relations.

\begin{definition}\label{def:leq}
    Let $\cX$ and $\cY$ be association schemes on the same set of points. We say that $\cX\leq \cY$ if and only if every relation 
    of $\cX$ is a union of some relations of $\cY$.
If this is the case, then $\cX$ is said to be a \emph{\textbf{fusion}}
(scheme) of $\cY$ and $\cY$ is a \emph{\textbf{fission}} (scheme) of $\cX$.
\end{definition}

The Bose-Mesner algebra of a fusion scheme $\cX$ is 
a linear subspace of the Bose-Mesner algebra of $\cY$. This may serve as an alternative definition of 
a fusion \cite[Section~5.2]{Godsil}; note that what we call 
a fusion scheme is sometimes called a subscheme~\cite{Godsil}.

Note that the relation $\leq$ is a partial order on the set 
of association schemes defined on the same point set $\Omega$.
For example, the association scheme of the complete graph 
on $\Omega$ is a fusion of any scheme on $\Omega$; 
in fact, it is the smallest scheme on $\Omega$. 

\begin{lemma}\label{lm:fusion}
Let $\Gamma$ denote the sum-rank metric graph 
$\Gamma\left(\mathbb{F}_q^{{\bf n}\times {\bf m}}\right)$ and 
suppose that $\cX_i$ is the bilinear forms scheme on $\F_q^{n_i\times m_i}$, $i\in [t]$.
    Then the adjacency relation of~$\Gamma$ is a union of 
    some relations of the direct product $\cX_1\otimes\cdots\otimes\cX_t$.
\end{lemma}
\begin{proof}
It follows directly from Proposition~\ref{prop:prod}, Eq.\:\eqref{eq:Acartprod}, and Definition \ref{def:direct}.
\end{proof}


\subsection{An association scheme for sum-rank metric graphs}\label{ssect:as-srk}
We now define the association scheme $\cX$ of the sum-rank metric 
graph $\Gamma$ as the smallest association scheme such that 
a union of some of its relations coincides with 
the adjacency relation of $\Gamma$ 
(indeed, such a scheme exists by Lemma \ref{lm:fusion}).
Again, in the more general context of 
coherent configurations \cite[Section~2.6.1]{CP2019}, 
$\cX$ is the \textbf{coherent closure} (or 
\textbf{Weisfeiler-Leman closure}) of the graph $\Gamma$, denoted $\WL(\Gamma)$, which can be computed by using the Weisfeiler-Leman algorithm~\cite{WL}. 
From the above discussion and Lemma~\ref{lm:fusion}, 
one has the following immediate consequence.

\begin{corollary}\label{cor:srk_leq}
    Let $\Gamma$ be a sum-rank metric graph, and let $\Gamma_i$, $i\in [t]$, be the $t$ bilinear forms 
    graphs that are the Cartesian factors of $\Gamma$. 
    Then 
    \begin{equation}\label{eq:leq}
\cX:=\WL(\Gamma) = \WL(\Gamma_1\Box\cdots\Box\Gamma_t) \leq 
\WL(\Gamma_1)\otimes \WL(\Gamma_2)\otimes\cdots\otimes\WL(\Gamma_t),
    \end{equation}
    moreover, $\WL(\Gamma_i)$ is a bilinear forms scheme, 
    $i\in [t]$.
\end{corollary}

The relation ``\emph{to be at distance $\ell$ 
in $\Gamma$}'' is the union of all 
relations $R_{(j_1,\ldots,j_t)}$ such that 
$j_1+j_2+\cdots+j_t=\ell$. 
Therefore, we can get an upper bound on the size 
of a sum-rank metric code with prescribed minimum distance $d$
by applying Delsarte's LP bound to a code in the direct 
product of bilinear forms schemes. 
Namely, we solve the LP\:\eqref{eq:LP} 
with the matrix $Q$ given by Lemma \ref{l:Q} 
and Eq.\:\eqref{eq:Qbil} and the distribution vector 
$\mathbf{a}=\left(a_{(j_1,\ldots,j_t)}\right)$ 
indexed by $\db{D_1} \times\cdots\times \db{D_t}$:
\begin{equation}\label{eq:LPsrk}
\boxed{
\begin{array}{ll@{}ll}
\text{maximize}  &\sum_{i\in \db{D_1} \times\cdots\times \db{D_t}} a_i &\\
\text{subject to} &\mathbf{a}Q\geq \mathbf{0},\\
&\mathbf{a}\geq \mathbf{0},\\
&a_{(0,\ldots,0)}=1,\\
&a_{(j_1,\ldots,j_t)} = 0, &\quad j_1+j_2+\cdots+j_t<d.
\end{array}
}
\end{equation}



It may happen that the scheme on the right-hand side of 
Eq.\:\eqref{eq:leq} is strictly larger than the one 
on the left-hand side, which results in 
the LP problem of larger size and brings us 
to the next subsection.

\subsection{On the coherent closure of a Cartesian product of graphs}\label{sec:ASsrkgraph}
The question when the inequality 
in Eq.\:\eqref{eq:leq} is strict is a subtle one.
This happens precisely when the direct product on the right-hand side 
of the formula admits a non-trivial fusion containing 
the edge set of the graph in question as a relation. 
In the general case, not restricted to bilinear forms graphs, 
the situation is studied in \cite{CGGP} 
using the theory of coherent configurations. 
In this section, we give a partial result in this direction using spectral graph theory. 

First, we recall a simple example which shows that 
the inequality in Eq.\:\eqref{eq:leq} can be strict.
Indeed, 
the Hamming association scheme on $\F_{q^n}$ 
is a non-trivial fusion 
of the direct product of the (trivial) schemes of complete graphs.
One can easily see this, as the former scheme has rank $n+1$, 
while the latter one has rank $2^n$.
To explain this phenomenon, 
we need to recall the following definition.

Let $\cX=(\Omega,\sR)$, $\sR=\{R_i\mid i\in \db{D}\}$, and
$\cY=(\Omega,\sT)$, $\sT=\{T_i\mid i\in \db{D}\}$, 
be two association schemes. 
Any permutation 
$\varphi\in \sym(\db{D})$, mapping $i\mapsto i^\varphi$,  
naturally induces a bijection from $\sR$ to $\sT$,  
which sends $R_i$ to $T_{i^\varphi}$. 
With a slight abuse of notation, 
we also denote the latter bijection by $\varphi$.
Now, if $\varphi\in \sym(\db{D})$ 
preserves the numbers from 
the condition~\ref{itm:interray_as} of Definition \ref{def:as} 
(namely, the intersection numbers $p_{i,j}^k$ of $\cX$ 
and $p_{i^{\varphi},j^{\varphi}}^{k^{\varphi}}$ of $\cY$
are equal for all $i,j,k\in \db{D}$), then the induced 
bijection $\varphi\colon \sR\to \sT$ 
is called an \textbf{algebraic isomorphism} 
from $\cX$ to $\cY$. 
If $\cX=\cY$, then $\varphi$ is  
an \textbf{algebraic automorphism} of $\cX$.
Clearly, algebraic automorphisms form a group under composition.

Let $\Phi$ be a group of algebraic automorphisms of $\cX$. 
Given $i\in \db{D}$, put
\[
R_i^{\Phi}=\cup_{\varphi\in\Phi}R_{i^\varphi}.
\]
One can see that the set $\sR^\Phi=\{R_i^\Phi\mid i\in \db{D}\}$ forms a partition of $\Omega^2$.
Furthermore, it follows (see \cite[Lemma~2.3.26]{CP2019}, 
\cite[Lemma~9.1.1]{Godsil}) that the pair $\cX^\Phi=(\Omega,\sR^\Phi)$ 
is an association scheme, called the \textbf{algebraic fusion} 
of $\cX$ with respect to $\Phi$.

Consider now the $t$-fold direct product:
\[
\cX^t=
\underbrace{\cX\otimes\cX\otimes\cdots\otimes\cX}_{t},\quad t\geq 2. 
\]
One can show (see \cite[Section~2.3]{CP2019}) 
that any permutation on $t$ symbols induces an 
algebraic automorphism of $\cX^t$; thus, there exists 
a non-trivial algebraic fusion of $\cX^t$, 
called the \textbf{exponentiation} of $\cX$ 
by the symmetric group $\sym(t)$.
Furthermore, when an association scheme $\cX$ is the coherent 
closure of a graph $\Gamma$, these algebraic automorphisms 
of $\cX^t$ leave invariant the adjacency relation of 
the Cartesian $t$-th power 
$\Gamma^{\Box t}:=\Gamma\Box \cdots\Box \Gamma$ 
of the graph $\Gamma$. Thus, the coherent closure of 
$\Gamma^{\Box t}$ is strictly less than $\cX^t$. 
For example, the Hamming scheme on $\F_{q^n}$ is the exponentiation of the trivial scheme on $q$ points 
by the symmetric group $\sym(n)$.

Therefore, when a Cartesian product of graphs contains 
isomorphic factors, its coherent closure is definitely 
smaller than the direct product of the coherent closures 
of the factors. 
In particular, this applies to a sum-rank metric graph 
in Eq.\:\eqref{eq:leq}.
Otherwise, we only have the following conjecture.

\begin{problem}\label{conj}
    Let $\Gamma$ be a Cartesian product of $t$ bilinear 
    forms graphs $\Gamma_i$, $i\in [t]$. 
    Suppose that the coherent closures 
    $\WL(\Gamma_i)$, $\WL(\Gamma_j)$
    are not algebraically isomorphic\footnote{Two bilinear forms schemes on $\F_q^{n_1\times m_1}$ and 
    $\F_q^{n_2\times m_2}$ are (algebraically) 
    isomorphic if and only if $\{n_1,m_1\}=\{n_2,m_2\}$.}
    for all $i\ne j\in [t]$.
    Is it true that
    then 
    \begin{equation}\label{eq:conjecture}
\WL(\Gamma_1\Box\cdots\Box\Gamma_t) = 
\WL(\Gamma_1)\otimes \WL(\Gamma_2)\otimes\cdots \otimes \WL(\Gamma_t)?
    \end{equation}    
\end{problem}

In other words, Eq.\:\eqref{eq:conjecture} states that 
the direct product on the right-hand side does not admit 
a non-trivial fusion containing the edge set of the graph 
in question as a relation. Note that a general 
necessary and sufficient existence condition of fusions 
is given by the Bannai-Muzychuk criterion \cite[Lemma~2.48]{BBIT}: see also Remark \ref{rem:fusion}.
In Theorem \ref{thm:WLeq} below, we prove a result
which shows the validity of Eq.\:\eqref{eq:conjecture}
under a stronger assumption than that in Problem  
\ref{conj}.

Note that the above discussion applies to the general case 
when the coherent closure of a graph is not an association 
scheme. Moreover, instead of the $t$-fold direct product 
one can considers a direct product of algebraically 
isomorphic schemes~\cite{CGGP}. 
For example, the Doob and Hamming association schemes on $\F_4^{2s+t}$ are algebraically isomorphic (for all natural $s,t$); 
both are the coherent closures of the Cartesian products 
of graphs: a Doob graph 
${Sh}^{\Box s}\Box K_4^{\Box t}$ in 
the former case (here ${Sh}$ stands for the Shrikhande graph), and a Hamming graph $K_4^{\Box (2s+t)}$ in the latter case. Note also that, 
except for the obvious trivial cases,  
both association schemes are strictly less than the direct products 
of the schemes of factors in the corresponding Cartesian products.

Next, we prove a sufficient condition under which 
the coherent closure of a Cartesian product of two graphs equals 
the direct product of their coherent closures.  
The condition is formulated in terms of the eigenvalues 
of the adjacency matrices of the graphs.
Recall that the \textbf{adjacency algebra} of a graph 
is a coherent algebra (see Section \ref{ssec:as}) generated 
by its adjacency matrix. 
In the proof of Theorem \ref{thm:WLeq} below, 
the reader may assume that the coherent closures 
$\WL{(\Gamma_1\Box \Gamma_2)}$, $\WL{(\Gamma_1)}$, $\WL{(\Gamma_2)}$ 
are association schemes, in which case the 
adjacency algebras of the graphs $\Gamma_1\Box \Gamma_2$, $\Gamma_1$, $\Gamma_2$ are simply the Bose-Mesner algebras (see Section \ref{ssec:as}) of the corresponding schemes.
However, the proof remains valid in the general case 
when these coherent closures are coherent configurations; 
the corresponding adjacency algebras are coherent algebras 
associated with coherent configurations \cite[Section~2.3.1]{CP2019}.

\begin{theorem}\label{thm:WLeq}
    Let $\Gamma_i$, $i\in\{1,2\}$, be a graph with precisely $s_i$ pairwise distinct eigenvalues $\theta_{ij}$, $j\in [s_i]$. Then 
    \[
    \WL{(\Gamma_1\Box \Gamma_2)}=\WL(\Gamma_1)\otimes \WL(\Gamma_2)
    \]
    if the set $S:=\{\theta_{1j}+\theta_{2k}\mid j\in [s_1], k\in [s_2]\}$ is of cardinality $s_1s_2$.
\end{theorem}


\begin{proof}
As in Corollary \ref{cor:srk_leq}, one can see that 
$\WL{(\Gamma_1\Box \Gamma_2)}\leq \WL(\Gamma_1)\otimes \WL(\Gamma_2)$. 
Therefore, to show the result, we need to prove 
$\WL{(\Gamma_1\Box \Gamma_2)}\geq 
\WL(\Gamma_1)\otimes \WL(\Gamma_2)$. Note that the fact that $\WL(\Gamma_1)\otimes \WL(\Gamma_2)$ is a fusion of $\WL(\Gamma_1\Box\Gamma_2)$ implies that the adjacency algebra $\mathcal{A}$ of $\Gamma_1\Box \Gamma_2$ contains the coherent (Bose-Mesner) algebra of $\WL(\Gamma_1)\otimes \WL(\Gamma_2)$. Combining this and \cite[Theorem~3.2.25]{CP2019}, we conclude that it is sufficient to prove that $\mathcal{A}$ contains the tensor product of the adjacency algebras $\mathcal{A}_1$ and $\mathcal{A}_2$ of $\Gamma_1$ and $\Gamma_2$, respectively.

To prove the latter statement, it in turn suffices to 
prove that 
\begin{equation}\label{eq:A1A2}
    A_1\otimes I_{n_2}\in \mathcal{A}\quad\text{or}\quad I_{n_1}\otimes A_2\in \mathcal{A},
\end{equation}
where $A_1,A_2$ denote the adjacency matrices of the graphs $\Gamma_1,\Gamma_2$ on $n_1,n_2$ vertices, respectively.
Indeed, it is sufficient to prove only one of the two inclusions in Eq.\:\eqref{eq:A1A2}: the adjacency algebras $\mathcal{A},\mathcal{A}_1,\mathcal{A}_2$ 
are generated (as coherent algebras) by the matrices $A=A_1\otimes I_{n_2}+I_{n_1}\otimes A_2$~(see Eq.\:\eqref{eq:Acartprod}), $A_1$, $A_2$, respectively.
Assume, for example, that $\mathcal{A}$ contains $A_1\otimes I_{n_2}$ and hence $\mathcal{A}$ contains the subalgebra $\mathcal{A}_1\otimes I_{n_2} = \{X\otimes I_{n_2}\,\colon\,X\in\mathcal{A}_1\}$.
Since $A\in\mathcal{A}$, it follows that $\mathcal{A}$ contains $A-A_1\otimes I_{n_2}=I_{n_1}\otimes A_2$ and hence the subalgebra $I_{n_1}\otimes \mathcal{A}_2 = \{I_{n_1}\otimes X\,\colon\,X\in\mathcal{A}_2\}$.
Thus, $\mathcal{A}$ contains both $\mathcal{A}_1\otimes I_{n_2}$ and $I_{n_1}\otimes \mathcal{A}_2$, and hence
\[
\mathcal{A}\supseteq 
\left(\mathcal{A}_1\otimes I_{n_2}\right)\cdot 
\left(I_{n_1}\otimes \mathcal{A}_2\right)=
\mathcal{A}_1\otimes\mathcal{A}_2,
\]
as required. 

Next, to prove Eq.\:\eqref{eq:A1A2}, we first note that 
\[
A^k=\sum_{i=0}^k {k\choose i}A_1^i\otimes A_2^{k-i}, 
\]
which follows from Eq. \eqref{eq:Acartprod} and 
the standard properties of Kronecker product of matrices.


Let $f_i$, $i\in \{1,2\}$, be the minimal polynomial 
of $A_i$ and let $I$ denote the ideal of 
the ring of polynomials $\mathbb{R}[x,y]$, generated by $f_1(x)$ and $f_2(y)$. 
Then the mappings $A_1\otimes I_{n_2}\mapsto x$ 
and $I_{n_1}\otimes A_2\mapsto y$, define 
an isomorphism, say $\iota$, between $\mathcal{A}$ and 
the subalgebra $\mathcal{K}$ of $\mathbb{R}[x,y]/I$ 
generated by the polynomial $x+y$, 

For $i\in\{1,2\}$, since $f_i$ is the minimal polynomial of $A_i$, its degree is $s_i$. It follows that the algebra $\mathbb{R}[x,y]/I$ is of dimension $N=s_1s_2$ and has a basis consisting of $x^iy^j$, $i\in\db{s_1-1}, j\in\db{s_2-1}$. In particular, the dimension of $\mathcal{K}$ is at most $N$.

Consider a linear mapping $\epsilon\colon \mathcal{K}\to \mathbb{R}^{N}$ 
defined for $p\in \mathcal{K}$ as follows:
\[
\epsilon\colon p(x,y)\mapsto 
\left(
p(\theta_{11},\theta_{21}), 
p(\theta_{11},\theta_{22}), 
\ldots, 
p(\theta_{1s_1},\theta_{2s_2})
\right).
\]
Let $K=\langle \epsilon(1),\epsilon(x+y),\epsilon((x+y)^2),\ldots,\epsilon((x+y)^{N-1})\rangle$ 
be a subspace of $\mathbb{R}^N$, the image of $\epsilon$. Due to the Rank-Nullity theorem, the dimension of $\mathcal{K}$ is lower-bounded by the dimension of the image of $\epsilon$.
Moreover, $\dim K$ equals the rank $r$ of the following $N\times N$ matrix:
\[
\left[
\begin{array}{ccc}
   1  & \ldots & 1 \\
   \theta_{11}+\theta_{21} & \ldots & \theta_{1s_1}+\theta_{2s_2}\\
   (\theta_{11}+\theta_{21})^2 & \ldots & (\theta_{1s_1}+\theta_{2s_2})^2\\
   \vdots & \vdots & \vdots
\end{array}
\right],
\]
which is a Vandermonde matrix. Hence $r$ equals the number of pairwise distinct numbers in the set $S$.
Therefore, if $|S|=N$, then the dimension of $\mathcal{K}$ is exactly $N$, and the elements $1,(x+y),\ldots,(x+y)^{N-1}$ constitute another basis of $\mathcal{K}$.
Thus, the elements of one basis $x^iy^j$, $i\in\db{s_1-1}$, $j\in\db{s_2-1}$, can be expressed through the elements of another basis $(x+y)^i,\,i\in\db{N-1}$. In particular,
$$x = \sum\limits_{i\in \db{N-1}} \alpha_i (x+y)^i$$
for some $\alpha_i\in \mathbb{R}$, $i\in\db{N-1}$. By applying $\iota^{-1}$, we get
$$A_1\otimes I_{n_2} = \sum\limits_{i\in \db{N-1}} \alpha_i A^i,$$
which implies $A_1\otimes I_{n_2}\in\mathcal{A}$, as required by Eq.\:\eqref{eq:A1A2}.
\end{proof}

\begin{remark}\label{rem:fusion}
For a sum-rank metric graph $\Gamma=\Gamma(\F^{\bfn\times\bfm}_q)$, 
which is the Cartesian product of $t$ bilinear forms 
graphs $\Gamma_i = \Bil_q(n_i,m_i)$, $i\in[t]$, 
the set of distinct eigenvalues of~$\Gamma$ is 
contained in the following set:  
$$S=\left\{ \sum\limits_{i = 1}^t \frac1{q-1} \left( (q^{n_i-j_i}-1)(q^{m_i}-q^{j_i})- q^{j_i} + 1 \right) \colon 0\leq j_i \leq n_i, i\in[t] \right\}.$$
Theorem~\ref{thm:WLeq} states that $\WL(\Gamma)$ is equal 
to the direct product of~$\WL(\Gamma_i)$, $i\in[t]$, 
if $S$ has $(n_1+1)\cdots(n_t+1)$ distinct elements.
One can easily find instances when this condition fails. 
For example, in case $t=2$, 
two eigenvalues corresponding to distinct pairs   
$(j_1,j_2)$ and $(j_1',j_2')$, $0\leq j_i,j_i'\leq n_i, i=1,2$, 
coincide if and only if 
$n_2-n_1+m_2-m_1=j_2-j_1'=j_2'-j_1$. 

Despite this, to support the affirmative answer to Problem \ref{conj}, 
we checked the Bannai-Muzychuk criterion for 
the direct product on the right-hand side of 
Eq. \eqref{eq:conjecture} 
for all sum-rank metric graphs over $\mathbb{F}_2$ 
with $t=2$, $\sum_{i\in [t]}m_i\leq 6$, and  
$t=3$, $\sum_{i\in [t]}m_i\leq 11$, and found 
no counterexamples. 
\end{remark}

\section{Bounds performance}\label{sec:boundcomp}

Tables~\ref{tab:q=2}, \ref{tab:q=3}, and \ref{tab:q=4} list the non-distance regular sum-rank metric graphs (restricted in the number of vertices $|\Omega|$ and the number of blocks $t$) for which we computed the Delsarte's LP bound on $\alpha_{d-1}$ by solving the LP\:\eqref{eq:LPsrk}, as well as the other previously known bounds from~\cite{byrne2021fundamental} and the Ratio-type bound introduced in~\cite{AAA2024}. 
Based on these computational experiments on relatively small instances, it is observed that the obtained bounds outperform all
previously known bounds. While we omit the cases when the result of the Delsarte's LP bound coincides with one of the previously achieved bounds, there are no examples of sum-rank metric graphs under the specified restrictions on $|\Omega|$ and $t$ for which the Delsarte's LP bound returns a result that is strictly larger than the result of one of the other upper bounds.

The column ``RT$_{d-1}$'' gives the values of the Ratio-type bound from Theorem~\ref{thm:hoffman-like}, which is calculated either explicitly in case $d\in\{3,4\}$ using the results of Theorem~\ref{thm:hoffman-k=2}~\cite{ACF2019} and Theorem~\ref{thm:hoffman-k=3}~\cite{KN2022}, or using the LP~\eqref{eq:FiolLP}~\cite{F2020}.
The columns ``D$_d$'', ``iS$_d$'', ``iH$_d$'', ``iE$_d$'', ``S$_d$'', ``SP$_d$'', ``PSP$_d$'' give the values of the Delsarte's LP\:\eqref{eq:FiolLP}, Induced Singleton, Induced Hamming, Induced Elias, Singleton, Sphere-Packing, and Projective Sphere-Packing bounds, respectively (see Theorems~\ref{thm:induced} and \ref{thm:non-induced}).

Since the code cardinality is always an integer, the columns only contain integer values, which are sometimes obtained by taking the floor of the real value given by the bound.

Note that the sum-rank metric graphs considered in Tables~\ref{tab:q=2}, \ref{tab:q=3}, and \ref{tab:q=4} are restricted not only in the size of the vertex set of the graph, but also in $t$, the number of blocks. Since calculating the Delsarte's LP bound requires taking a Kronecker product of $t$ matrices, the computation speed of the bound changes significantly with the increase of $t$. This obstacle is avoided when using the other bounds presented, particularly the Ratio-type bound, which, unlike the Delsarte's LP bound, does not depend as much on the structure of the graph in its design. Hence the Ratio-type bound can be useful in case $t$ is large, despite it being often outperformed by Delsarte's LP bound for smaller values of $t$. The Ratio-type bound can also be computed without the use of LP for $d=3,4$ from explicit formul{\ae}~\cite{ACF2019,KN2022,AAA2024}, which also puts it at advantage in this case.

To conclude this section, we discuss the relationship between the Delsarte's LP bound and the Lov\'asz theta number. For a graph $\Gamma$, we define $\Gamma^k$ to be a \textbf{$k$-th power} of $\Gamma$ if the vertex set of $\Gamma^k$ is the same as the vertex set of $\Gamma$, and two vertices are adjacent in $\Gamma^k$ if and only if they are at distance at most $k$ in $\Gamma$. It is then clear that the $k$-independence number of~$\Gamma$ is exactly the independence number of $\Gamma^k$. It is well-known that the independence number is upper bounded by the Lov\'asz theta number, denoted by $\vartheta$~\cite{Lovasz}. Thus we use $\vartheta_k$ to denote the Lov\'asz theta number of $\Gamma^k$, which is an upper bound on $\alpha_k$. The number $\vartheta_k$ can be estimated using Semidefinite Programming (SDP) as follows~\cite{Lovasz}: 
Let $\Gamma=(\Omega,E)$ be a graph on $n$ vertices, and let $\mathcal{S}_+^n$ denote the set of all $n\times n$ symmetric positive semidefinite matrices. Then $\vartheta$ is the solution of:
\begin{equation}\label{eq:LovaszSDP}
\boxed{
\begin{array}{ll@{}l}
\text{maximize}  &\operatorname{tr}(BJ) &\\
\text{subject to} &\operatorname{tr}B=1, &\\
&b_{ij}=0, &\quad (i,j)\in E,\,i,j\in[n],\\
&B\in\mathcal{S}_+^n, &\\ 
\end{array}
}
\end{equation}
where $J$ is an all-ones matrix of size $n\times n$. Note that $\operatorname{tr}(BJ)$ is the sum of the entries in $B$.

\begin{remark}
    In~\cite{S79}, it was shown that, for graphs derived from symmetric association schemes, the bound obtained through Delsarte's LP method~\cite{DelsarteLP} coincides with the number $\vartheta'$, which is calculated similarly to $\vartheta$ using SDP\:\eqref{eq:LovaszSDP}, with the additional constraint that $B$ must be a non-negative matrix. In the same paper~\cite{S79} it was proven that $\alpha \leq \vartheta' \leq \vartheta$, meaning that the Delsarte's LP bound performs no worse than the Lov\'asz theta bound. This is confirmed by computations on sum-rank metric graphs given in Table~\ref{tab:Lovasz} in Appendix.
\end{remark}

\section{Concluding remarks}\label{sec:final}
In this work we show how to implement Delsarte's LP method to obtain an upper bound on the maximum size of a code with a given minimum distance in sum-rank metric, which does not correspond to a distance-regular graph (unlike the Hamming metric or rank metric). We do so by considering the direct product of the association schemes of the Cartesian factors of the corresponding sum-rank metric graph.
Our results mostly make use of the structure of the space itself and not of the properties of matrix ranks that define the distance between the elements. Thus, we expect that this approach for constructing association schemes could also be applied to other metrics that give rise to graphs with a similar structure.
In particular, most results of Section~\ref{sec:ASsrkgraph} are provided in a general form and could be applied to any Cartesian product of graphs instead of a sum-rank metric graph.
For instance, codes in the Lee-metric can be viewed as subsets 
of vertices in the standard path-length metric of 
the Cartesian products of cycles; note that in \cite{A1982} the linear programming bound for Lee-codes 
was obtained by applying Delsarte's LP method to the explicitly defined 
Lee association scheme \cite{S1986}.

The application of Delsarte's LP bound in case the graph corresponding to the distance metric is distance-regular also gives rise to some open problems. Consider the rank-metric case, \textit{i.e.} the case of sum-rank metric graphs with $t=1$. The computational experiments on small graphs (up to $10^7$ vertices with diameter at least~$3$) shown in Table~\ref{tab:t=1} suggest that for such graphs, both Delsarte's LP bound computed with the LP\:\eqref{eq:LPsrk} and the Ratio-type bound of Theorem~\ref{thm:hoffman-like} computed with the LP \eqref{eq:FiolLP}  return the same result (this concerns the bounds on the size of the maximum code with minimum distance $d\in[3,n]$). Such output also coincides with the Singleton bound, which is known to be always met in the sum-rank metric case~\cite{D1978}. Thus, we suspect that the first two bounds might be equivalent in the case of rank-metric graphs.

\subsection*{Acknowledgements} 
Aida Abiad is supported by NWO (Dutch Research Council) through the grant VI.Vidi.213.085. The research of Alexander Gavrilyuk is supported 
by JSPS KAKENHI Grant Number 22K03403. Antonina P. Khramova is supported by NWO via the grant OCENW.KLEIN.475. This work was initiated during the RICCOTA conference, Croatia, 
in July 2023; the authors would like to thank 
the organizers of the event.

\bibliographystyle{abbrv}
\bibliography{references}

\newpage
\section*{Appendix}

\begin{table}[!htbp]
\centering
{\tiny
\[
\begin{array}{|ccllc|c|cc|ccc|ccc|}
\hline
t & q & {\bf n} & {\bf m} & d & |\Omega| & 	\text{RT}_{d-1} & 	\text{D}_d & 	\text{iS}_d  & 	\text{iH}_d & 	\text{iE}_d  & 	\text{S}_d  & 	\text{SP}_d  & 	\text{PSP}_d \\ \hline
2 & 2 & [2, 2] & [2, 2] & 3 & 256 & 11~\text{\cite{ACF2019}} & 10 & 16 & 19 & 34 & 16 & 13 & 13 \\
3 & 2 & [2, 2, 1] & [2, 2, 1] & 3 & 512 & 25~\text{\cite{ACF2019}} & 20 & 64 & 64 & 151 & 32 & 25 & 25 \\
3 & 2 & [2, 2, 1] & [2, 2, 1] & 4 & 512 & 10~\text{\cite{KN2022}} & 6 & 16 & 64 & 27 & 8 & 25 & 18 \\
3 & 2 & [2, 2, 1] & [2, 2, 2] & 3 & 1024 & 38~\text{\cite{ACF2019}} & 34 & 64 & 64 & 151 & 64 & 46 & 46 \\
3 & 2 & [2, 2, 1] & [2, 2, 2] & 4 & 1024 & 15~\text{\cite{KN2022}} & 8 & 16 & 64 & 27 & 16 & 46 & 36 \\
3 & 2 & [2, 2, 2] & [2, 2, 2] & 3 & 4096 & 146~\text{\cite{ACF2019}} & 133 & 256 & 215 & 529 & 256 & 146 & 146 \\
3 & 2 & [2, 2, 2] & [2, 2, 2] & 4 & 4096 & 45~\text{\cite{KN2022}} & 32 & 64 & 215 & 119 & 64 & 146 & 102 \\
3 & 2 & [2, 2, 2] & [2, 2, 2] & 5 & 4096 & 14~\text{\cite{F2020}} & 6 & 16 & 26 & 19 & 16 & 14 & 13 \\
4 & 2 & [2, 1, 1, 1] & [2, 2, 2, 1] & 3 & 512 & 28~\text{\cite{ACF2019}} & 24 & 64 & 64 & 151 & 32 & 30 & 30 \\
4 & 2 & [2, 1, 1, 1] & [2, 2, 2, 1] & 4 & 512 & 11~\text{\cite{KN2022}} & 6 & 16 & 64 & 27 & 8 & 30 & 32 \\
4 & 2 & [2, 1, 1, 1] & [2, 2, 2, 2] & 3 & 1024 & 44~\text{\cite{ACF2019}} & 42 & 64 & 64 & 151 & 64 & 53 & 53 \\
4 & 2 & [2, 1, 1, 1] & [2, 2, 2, 2] & 4 & 1024 & 18~\text{\cite{KN2022}} & 10 & 16 & 64 & 27 & 16 & 53 & 64 \\
4 & 2 & [2, 2, 1, 1] & [2, 2, 1, 1] & 3 & 1024 & 46~\text{\cite{ACF2019}} & 40 & 256 & 215 & 529 & 64 & 48 & 48 \\
4 & 2 & [2, 2, 1, 1] & [2, 2, 1, 1] & 4 & 1024 & 19~\text{\cite{KN2022}} & 12 & 64 & 215 & 119 & 16 & 48 & 36 \\
4 & 2 & [2, 2, 1, 1] & [2, 2, 2, 1] & 3 & 2048 & 85~\text{\cite{ACF2019}} & 68 & 256 & 215 & 529 & 128 & 89 & 89 \\
4 & 2 & [2, 2, 1, 1] & [2, 2, 2, 1] & 4 & 2048 & 32~\text{\cite{KN2022}} & 16 & 64 & 215 & 119 & 32 & 89 & 73 \\
4 & 2 & [2, 2, 1, 1] & [2, 2, 2, 1] & 5 & 2048 & 12~\text{\cite{F2020}} & 5 & 16 & 26 & 19 & 8 & 10 & 9 \\
4 & 2 & [2, 2, 1, 1] & [2, 2, 2, 2] & 3 & 4096 & 146~\text{\cite{ACF2019}} & 134 & 256 & 215 & 529 & 256 & 163 & 163 \\
4 & 2 & [2, 2, 1, 1] & [2, 2, 2, 2] & 4 & 4096 & 48~\text{\cite{KN2022}} & 32 & 64 & 215 & 119 & 64 & 163 & 146 \\
4 & 2 & [2, 2, 1, 1] & [2, 2, 2, 2] & 5 & 4096 & 17~\text{\cite{F2020}} & 7 & 16 & 26 & 19 & 16 & 17 & 16 \\
5 & 2 & [2, 1, 1, 1, 1] & [2, 1, 1, 1, 1] & 5 & 256 & 5~\text{\cite{F2020}} & 2 & 16 & 26 & 19 & 4 & 4 & 3 \\
5 & 2 & [2, 1, 1, 1, 1] & [3, 1, 1, 1, 1] & 5 & 1024 & 8~\text{\cite{F2020}} & 2 & 64 & 336 & 240 & 4 & 6 & 3 \\
5 & 2 & [2, 1, 1, 1, 1] & [4, 1, 1, 1, 1] & 5 & 4096 & 11~\text{\cite{F2020}} & 2 & 256 & 4096 & 1970 & 4 & 9 & 3 \\
5 & 2 & [2, 1, 1, 1, 1] & [2, 2, 2, 1, 1] & 3 & 1024 & 56~\text{\cite{ACF2019}} & 49 & 256 & 215 & 529 & 64 & 56 & 56 \\
5 & 2 & [2, 1, 1, 1, 1] & [2, 2, 2, 1, 1] & 4 & 1024 & 22~\text{\cite{KN2022}} & 13 & 64 & 215 & 119 & 16 & 56 & 64 \\
5 & 2 & [2, 1, 1, 1, 1] & [2, 2, 2, 2, 1] & 3 & 2048 & 102~\text{\cite{ACF2019}} & 85 & 256 & 215 & 529 & 128 & 102 & 102 \\
5 & 2 & [2, 1, 1, 1, 1] & [2, 2, 2, 2, 1] & 4 & 2048 & 36~\text{\cite{KN2022}} & 20 & 64 & 215 & 119 & 32 & 102 & 128 \\
5 & 2 & [2, 1, 1, 1, 1] & [2, 2, 2, 2, 1] & 5 & 2048 & 14~\text{\cite{F2020}} & 6 & 16 & 26 & 19 & 8 & 13 & 11 \\
5 & 2 & [2, 1, 1, 1, 1] & [2, 2, 2, 2, 2] & 3 & 4096 & 153~\text{\cite{ACF2019}} & 146 & 256 & 215 & 529 & 256 & 186 & 186 \\
5 & 2 & [2, 1, 1, 1, 1] & [2, 2, 2, 2, 2] & 4 & 4096 & 53~\text{\cite{KN2022}} & 40 & 64 & 215 & 119 & 64 & 186 & 256 \\
5 & 2 & [2, 1, 1, 1, 1] & [2, 2, 2, 2, 2] & 5 & 4096 & 20~\text{\cite{F2020}} & 8 & 16 & 26 & 19 & 16 & 21 & 19 \\
5 & 2 & [2, 2, 1, 1, 1] & [2, 2, 1, 1, 1] & 3 & 2048 & 93~\text{\cite{ACF2019}} & 80 & 1024 & 744 & 1876 & 128 & 93 & 93 \\
5 & 2 & [2, 2, 1, 1, 1] & [2, 2, 1, 1, 1] & 4 & 2048 & 38~\text{\cite{KN2022}} & 24 & 256 & 744 & 407 & 32 & 93 & 73 \\
5 & 2 & [2, 2, 1, 1, 1] & [2, 2, 1, 1, 1] & 5 & 2048 & 13~\text{\cite{F2020}} & 7 & 64 & 77 & 99 & 8 & 11 & 9 \\
5 & 2 & [2, 2, 1, 1, 1] & [2, 2, 2, 1, 1] & 3 & 4096 & 170~\text{\cite{ACF2019}} & 137 & 1024 & 744 & 1876 & 256 & 170 & 170 \\
5 & 2 & [2, 2, 1, 1, 1] & [2, 2, 2, 1, 1] & 4 & 4096 & 64~\text{\cite{KN2022}} & 32 & 256 & 744 & 407 & 64 & 170 & 146 \\
5 & 2 & [2, 2, 1, 1, 1] & [2, 2, 2, 1, 1] & 5 & 4096 & 22~\text{\cite{F2020}} & 11 & 64 & 77 & 99 & 16 & 19 & 17 \\
6 & 2 & [2, 1, 1, 1, 1, 1] & [2, 1, 1, 1, 1, 1] & 4 & 512 & 16~\text{\cite{KN2022}} & 12 & 256 & 512 & 407 & 16 & 34 & 32 \\
6 & 2 & [2, 1, 1, 1, 1, 1] & [2, 1, 1, 1, 1, 1] & 5 & 512 & 8~\text{\cite{F2020}} & 4 & 64 & 77 & 99 & 8 & 6 & 5 \\
6 & 2 & [2, 1, 1, 1, 1, 1] & [3, 1, 1, 1, 1, 1] & 5 & 2048 & 15~\text{\cite{F2020}} & 4 & 512 & 1943 & 1707 & 8 & 11 & 5 \\
6 & 2 & [2, 1, 1, 1, 1, 1] & [3, 1, 1, 1, 1, 1] & 6 & 2048 & 11~\text{\cite{F2020}} & 2 & 64 & 1943 & 211 & 4 & 11 & 3 \\
6 & 2 & [2, 1, 1, 1, 1, 1] & [2, 2, 1, 1, 1, 1] & 5 & 1024 & 11~\text{\cite{F2020}} & 6 & 64 & 77 & 99 & 8 & 9 & 8 \\
6 & 2 & [2, 1, 1, 1, 1, 1] & [2, 2, 1, 1, 1, 1] & 6 & 1024 & 7~\text{\cite{F2020}} & 2 & 16 & 77 & 14 & 4 & 9 & 3 \\
6 & 2 & [2, 1, 1, 1, 1, 1] & [3, 2, 1, 1, 1, 1] & 6 & 4096 & 16~\text{\cite{F2020}} & 2 & 64 & 1943 & 211 & 4 & 17 & 3 \\
6 & 2 & [2, 1, 1, 1, 1, 1] & [2, 2, 2, 1, 1, 1] & 3 & 2048 & 102~\text{\cite{ACF2019}} & 99 & 1024 & 744 & 1876 & 128 & 107 & 107 \\
6 & 2 & [2, 1, 1, 1, 1, 1] & [2, 2, 2, 1, 1, 1] & 4 & 2048 & 42~\text{\cite{KN2022}} & 26 & 256 & 744 & 407 & 32 & 107 & 128 \\
6 & 2 & [2, 1, 1, 1, 1, 1] & [2, 2, 2, 2, 1, 1] & 3 & 4096 & 186~\text{\cite{ACF2019}} & 170 & 1024 & 744 & 1876 & 256 & 195 & 195 \\
6 & 2 & [2, 1, 1, 1, 1, 1] & [2, 2, 2, 2, 1, 1] & 4 & 4096 & 70~\text{\cite{KN2022}} & 40 & 256 & 744 & 407 & 64 & 195 & 256 \\
6 & 2 & [2, 1, 1, 1, 1, 1] & [2, 2, 2, 2, 1, 1] & 5 & 4096 & 26~\text{\cite{F2020}} & 12 & 64 & 77 & 99 & 16 & 23 & 21 \\
6 & 2 & [2, 2, 1, 1, 1, 1] & [2, 2, 1, 1, 1, 1] & 3 & 4096 & 170~\text{\cite{ACF2019}} & 160 & 4096 & 2621 & 4096 & 256 & 178 & 178 \\
6 & 2 & [2, 2, 1, 1, 1, 1] & [2, 2, 1, 1, 1, 1] & 4 & 4096 & 72~\text{\cite{KN2022}} & 48 & 1024 & 2621 & 1419 & 64 & 178 & 146 \\
6 & 2 & [2, 2, 1, 1, 1, 1] & [2, 2, 1, 1, 1, 1] & 5 & 4096 & 24~\text{\cite{F2020}} & 14 & 256 & 236 & 366 & 16 & 21 & 18 \\
6 & 2 & [2, 2, 1, 1, 1, 1] & [2, 2, 1, 1, 1, 1] & 6 & 4096 & 12~\text{\cite{F2020}} & 6 & 64 & 236 & 79 & 8 & 21 & 16 \\
6 & 2 & [2, 2, 1, 1, 1, 1] & [2, 2, 1, 1, 1, 1] & 7 & 4096 & 6~\text{\cite{F2020}} & 2 & 16 & 36 & 11 & 4 & 5 & 3 \\
\hline
\end{array}
\]
}
    \caption{Some of the non-distance-regular sum-rank metric graphs with $\leq5000$ vertices for which the LP~\eqref{eq:LPsrk} outperforms previously known bounds for $A_q({\bf n},{\bf m},d)$ with $t\leq6$.
    }
    \label{tab:q=2}
\end{table}

\begin{landscape}
\begin{table}[!htbp]
\centering
{\tiny
\[
\begin{array}{|ccllc|c|cc|ccc|ccc|}
\hline
t & q & {\bf n} & {\bf m} & d & |\Omega| & \text{RT}_{d-1} & \text{D}_d & \text{iS}_d  & \text{iH}_d & \text{iE}_d  & \text{S}_d  & \text{SP}_d  & \text{PSP}_d \\ \hline
3 & 3 & [2, 2, 2] & [2, 2, 2] & 3 & 531441 & 4652~\text{\cite{ACF2019}} & 4278 & 6561 & 10845 & 27399 & 6561 & 5478 & 5478 \\
3 & 3 & [2, 2, 2] & [2, 2, 2] & 5 & 531441 & 198~\text{\cite{F2020}} & 33 & 81 & 526 & 343 & 81 & 160 & 100 \\
6 & 3 & [2, 1, 1, 1, 1, 1] & [2, 1, 1, 1, 1, 1] & 5 & 19683 & 86~\text{\cite{F2020}} & 18 & 729 & 3413 & 2772 & 27 & 43 & 22 \\
6 & 3 & [2, 1, 1, 1, 1, 1] & [3, 1, 1, 1, 1, 1] & 5 & 177147 & 170~\text{\cite{F2020}} & 18 & 19683 & 177147 & 177147 & 27 & 97 & 22 \\
6 & 3 & [2, 1, 1, 1, 1, 1] & [3, 1, 1, 1, 1, 1] & 6 & 177147 & 148~\text{\cite{F2020}} & 6 & 729 & 177147 & 10104 & 9 & 97 & 9 \\
7 & 3 & [2, 1, 1, 1, 1, 1, 1] & [2, 1, 1, 1, 1, 1, 1] & 5 & 59049 & 223~\text{\cite{F2020}} & 48 & 6561 & 23180 & 26541 & 81 & 109 & 56 \\
7 & 3 & [2, 1, 1, 1, 1, 1, 1] & [2, 1, 1, 1, 1, 1, 1] & 6 & 59049 & 142~\text{\cite{F2020}} & 18 & 729 & 23180 & 3256 & 27 & 109 & 22 \\
7 & 3 & [2, 1, 1, 1, 1, 1, 1] & [3, 1, 1, 1, 1, 1, 1] & 5 & 531441 & 482~\text{\cite{F2020}} & 48 & 531441 & 531441 & 531441 & 81 & 259 & 56 \\
7 & 3 & [2, 1, 1, 1, 1, 1, 1] & [3, 1, 1, 1, 1, 1, 1] & 6 & 531441 & 406~\text{\cite{F2020}} & 18 & 19683 & 531441 & 531441 & 27 & 259 & 22 \\
7 & 3 & [2, 1, 1, 1, 1, 1, 1] & [2, 2, 1, 1, 1, 1, 1] & 6 & 177147 & 222~\text{\cite{F2020}} & 18 & 729 & 23180 & 3256 & 27 & 222 & 22 \\
8 & 3 & [2, 1, 1, 1, 1, 1, 1, 1] & [2, 1, 1, 1, 1, 1, 1, 1] & 6 & 177147 & 346~\text{\cite{F2020}} & 48 & 6561 & 162987 & 45526 & 81 & 282 & 56 \\
8 & 3 & [2, 1, 1, 1, 1, 1, 1, 1] & [2, 1, 1, 1, 1, 1, 1, 1] & 7 & 177147 & 158~\text{\cite{F2020}} & 12 & 729 & 8536 & 4625 & 27 & 41 & 22 \\
8 & 3 & [2, 1, 1, 1, 1, 1, 1, 1] & [2, 2, 1, 1, 1, 1, 1, 1] & 6 & 531441 & 562~\text{\cite{F2020}} & 48 & 6561 & 162987 & 45526 & 81 & 592 & 56 \\
8 & 3 & [2, 1, 1, 1, 1, 1, 1, 1] & [2, 2, 1, 1, 1, 1, 1, 1] & 7 & 531441 & 233~\text{\cite{F2020}} & 18 & 729 & 8536 & 4625 & 27 & 70 & 22 \\
9 & 3 & [2, 1, 1, 1, 1, 1, 1, 1, 1] & [2, 1, 1, 1, 1, 1, 1, 1, 1] & 5 & 531441 & 1471~\text{\cite{F2020}} & 340 & 531441 & 531441 & 531441 & 729 & 737 & 385 \\
9 & 3 & [2, 1, 1, 1, 1, 1, 1, 1, 1] & [2, 1, 1, 1, 1, 1, 1, 1, 1] & 6 & 531441 & 863~\text{\cite{F2020}} & 139 & 59049 & 531441 & 320840 & 243 & 737 & 145 \\
9 & 3 & [2, 1, 1, 1, 1, 1, 1, 1, 1] & [2, 1, 1, 1, 1, 1, 1, 1, 1] & 7 & 531441 & 362~\text{\cite{F2020}} & 37 & 6561 & 54141 & 32998 & 81 & 96 & 56 \\
\hline
\end{array}
\]
}
    \caption{Some of the non-distance-regular sum-rank metric graphs with $\leq10^6$ vertices for which the LP~\eqref{eq:LPsrk} outperforms previously known bounds for $A_q({\bf n},{\bf m},d)$ with $q\neq 2$.
    }
    \label{tab:q=3}
\end{table}

\begin{table}[!htbp]
\centering
{\tiny
\[
\begin{array}{|ccllc|c|cc|ccc|ccc|}
\hline
t & q & {\bf n} & {\bf m} & d & |\Omega| & \text{RT}_{d-1} & \text{D}_d & \text{iS}_d  & \text{iH}_d & \text{iE}_d  & \text{S}_d  & \text{SP}_d  & \text{PSP}_d \\ \hline
7 & 4 & [2, 1, 1, 1, 1, 1, 1] & [2, 1, 1, 1, 1, 1, 1] & 5 & 1048576 & 1740~\text{\cite{F2020}} & 179 & 65536 & 668893 & 549450 & 256 & 596 & 215 \\
8 & 4 & [2, 1, 1, 1, 1, 1, 1, 1] & [2, 1, 1, 1, 1, 1, 1, 1] & 5 & 4194304 & 6156~\text{\cite{F2020}} & 614 & 1048576 & 4194304 & 4194304 & 1024 & 2055 & 744 \\
8 & 4 & [2, 1, 1, 1, 1, 1, 1, 1] & [2, 1, 1, 1, 1, 1, 1, 1] & 6 & 4194304 & 4274~\text{\cite{F2020}} & 179 & 65536 & 4194304 & 1324989 & 256 & 2055 & 215 \\
8 & 4 & [2, 1, 1, 1, 1, 1, 1, 1] & [2, 1, 1, 1, 1, 1, 1, 1] & 7 & 4194304 & 1537~\text{\cite{F2020}} & 40 & 4096 & 235553 & 80457 & 64 & 200 & 64 \\
\hline
\end{array}
\]
}
    \caption{Some of the non-distance-regular sum-rank metric graphs with $\leq10^7$ vertices for which the LP~\eqref{eq:LPsrk} outperforms previously known bounds for $A_q({\bf n},{\bf m},d)$ with $q>3$.
    }
    \label{tab:q=4}
\end{table}
\end{landscape}

\begin{table*}[!t]
{
    \centering
    {\tiny
\[
\begin{array}{|ccllc|c|cccc|}
\hline
t & q & {\bf n} & {\bf m} & d & |\Omega| & \alpha_{d-1} & \vartheta_{d-1} & \text{RT}_{d-1} & \text{Delsarte's LP} \\
\hline
 2 & 2 &             [2, 1]            &             [2, 1]          & 3 &  32 &   2    &   2 &  2  &  2 \\
 2 & 2 &             [2, 1]            &             [2, 2]       & 3     &  64 &   4    &   4  & 4  & 4 \\
 2 & 2 &             [2, 1]            &             [3, 3]  & 3          & 512 &   8    &   8  & 8  & 8  \\  
 2 & 2 &             [2, 2]            &             [2, 2]     & 3       & 256  &  9   & 10 & 11    & 10 \\  
 2 & 3 &             [2, 1]            &             [2, 2]        & 3    & 729 &   9    &   9  & 9  & 9 \\
 3 & 2 &           [1, 1, 1]           &           [2, 1, 1]          & 3 &  16 &   2    &   2  & 2  & 2 \\
 3 & 2 &           [2, 1, 1]           &           [2, 2, 2] & 3          & 256 &   16   &   16  & 16 & 16 \\ 
 3 & 2 &           [2, 2, 1]           &           [2, 2, 1]    & 3       & 512   & 18    & 20 &  25   &  20 \\ 
 4 & 2 &          [1, 1, 1, 1]         &          [2, 1, 1, 1]     & 3    &  32 &   4    &   4   & 4 & 4 \\
  4 & 2 &          [1, 1, 1, 1]         &          [2, 1, 1, 1]    & 4     &  32 &   2    &   2  & 2   & 2 \\
 4 & 2 &          [2, 1, 1, 1]         &          [2, 1, 1, 1]    & 4     & 128 &   4    &   4  & 4  & 4 \\
 4 & 2 &          [2, 1, 1, 1]        &          [2, 2, 1, 1]         & 3 & 256 &   16   &   16 & 16  & 16 \\ 
 4 & 2 &          [2, 1, 1, 1]         &          [2, 2, 2, 1] & 3        & 512 & \text{time} & 24  & 28 & 24 \\  
 5 & 2 &        [1, 1, 1, 1, 1]        &        [2, 1, 1, 1, 1]   & 3   &  64 &   8  &   8          & 8  &  8 \\
 5 & 2 &        [2, 1, 1, 1, 1]        &        [2, 2, 1, 1, 1]      & 3  & 512 &   32   &   32 & 32 & 32 \\ 
 6 & 2 &     [1, 1, 1, 1, 1, 1] &  [2, 1, 1, 1, 1, 1]   & 3                & 128  &   9  & 12 & 12   & 12 \\
 6 & 2 &   [2, 1, 1, 1, 1, 1]    &    [2, 1, 1, 1, 1, 1]   & 3           & 512  &   \text{time}  &  32 & 32  & 32 \\
 7 & 2 &  [1, 1, 1, 1, 1, 1, 1]  &  [2, 1, 1, 1, 1, 1, 1]     & 3          & 256  &  \text{time}  & 25 &  25  & 21 \\ 
 8 & 2 &    [1, 1, 1, 1, 1, 1, 1, 1]   &    [2, 1, 1, 1, 1, 1, 1, 1] & 3   & 512  &  \text{time}  & 42 &  42  & 42 \\
\hline
\end{array}
\]
}
}
    \caption{Some sum-rank metric graphs on up to $1000$ vertices with the exact $(d-1)$-independence number $\alpha_{d-1}$, $\vartheta_{d-1}$, Ratio-type bound, and Delsarte's LP bound computed. We write `time' whenever the computation time exceeds $2$ hours on a standard laptop.
    }
    \label{tab:Lovasz}
\end{table*}

\vspace{-1.6cm}

\begin{table}[!htbp]
    \centering
    {\small
$$
\begin{array}{|cc|c|ccc|} \hline
\Gamma & d & |\Omega| & \text{Ratio-type} & \text{Delsarte's LP}  & \text{Singleton} \\ \hline
\Bil_2(3,3) & 3 & 512 & 8 & 8 & 8 \\
\Bil_2(3,4) & 3 & 4096 & 16 & 16 & 16 \\
\Bil_2(3,5) & 3 & 32768 & 32 & 32 & 32 \\
\Bil_2(3,6) & 3 & 262144 & 64 & 64 & 64 \\
\Bil_2(3,7) & 3 & 2097152 & 128 & 128 & 128 \\
\Bil_2(4,4) & 3 & 65536 & 256 & 256 & 256 \\
\Bil_2(4,4) & 4 & 65536 & 16 & 16 & 16 \\
\Bil_2(4,5) & 3 & 1048576 & 1024 & 1024 & 1024 \\
\Bil_2(4,5) & 4 & 1048576 & 32 & 32 & 32 \\
\Bil_3(3,3) & 3 & 19683 & 27 & 27 & 27 \\
\Bil_3(3,4) & 3 & 531441 & 81 & 81 & 81 \\
\Bil_4(3,3) & 3 & 262144 & 64 & 64 & 64 \\
\Bil_5(3,3) & 3 & 1953125 & 125 & 125 & 125 \\ \hline
\end{array}
$$
}
\caption{The list of rank-metric graphs, \textit{i.e.} sum-rank metric graphs with $t=1$, and the respective values of Delsarte's LP, Ratio-type, and Singleton bounds.}
    \label{tab:t=1}
\end{table}

\end{document}